\documentclass[12pt,a4paper]{amsart}
\usepackage{amsaddr}
\usepackage{amsmath,amscd}
\usepackage{latexsym,amsmath,amssymb,mathrsfs,setspace,enumerate,tikz}
\usepackage{graphicx}
\usepackage{subcaption}
\usepackage[T1]{fontenc}
\usepackage[utf8]{inputenc}
\usepackage{lmodern}
\usepackage{amssymb}
\usepackage{multicol}
\usepackage{tikz-network}
\usetikzlibrary{automata,positioning}
\usepackage[english]{babel} 
\usepackage{blindtext}

\theoremstyle{plain}
\newtheorem{theorem}{Theorem}[section]
\newtheorem{lemma}[theorem]{Lemma}
\newtheorem{proposition}[theorem]{Proposition}

\theoremstyle{definition}
\newtheorem{definition}[theorem]{Definition}

\newtheorem{example}[theorem]{Example}
 
\DeclareMathOperator{\dom}{dom}

\begin{document}
	
	\title{ Ideal category of a Noetherian ring}
	
			\author[P. G. Romeo]{P. G. Romeo and 
			Minnumol P K}
			
	\address{Department of Mathematics\\
		Cochin University of Science and Technology(CUSAT)\\Kochi, Kerala, India, 682022}
	\email{$romeo_-parackal@yahoo.com,\, pkminnumol@gmail.com$}
	
	\thanks{First author wishes to thank Council of Scientific and Industrial Research(CSIR) INDIA, for providing financial support.}
	
	\begin{abstract}
	In this paper we describe the categories $ \mathbb{L}_R,\, [\mathbb{R}_R]$  whose objects are left [right] ideals of a  Noetherian ring $ R $ with unity and morphisms are appropriate $R$-linear transformations. Further it is shown that these are preadditive categories with zero object and are full subcategories of the  $ R $ - modulue category with the property that these are categories with subobjects and the morphisms admits factorization property.		
	\end{abstract}
\maketitle
	\section{Introduction}\label{sec:intro}	
Category theory was introduced by Samuel Elienberg and Saunders Maclane\cite{maclane} in 1945 with the grant aim to unify the structural analysis of various stuctures and to simplify the   presentation. Several mathematicians successfully used category theory to study various mathematical structures. In  \cite{kss} K.S.S Nambooripad described certain categories which he call normal categories to characterize the ideals of a regular semigroup as categories which enabes him to obtain beautiful representations of regular semigroups. Later in 
\cite{pgr} this approach was extended to regular rings.
 In this paper we consider  a Noetherian ring $ R $ with unity and describe categories of 
 and right ideals of $ R $ having morphisms the left (right) $ R $ - linear transformations  and discuss various properties of these categories and compare these with the categorical properties of $ R $ - modulue categores. 
 
\section{Preliminaries}

	A category $ {\mathcal C} $ consisting of a class of  objects  written as $\nu {\mathcal C}$  and collection of morphisms $ f \in {\mathcal C(A,B)}$ from each object $ A = dom\; f $  to each object $ B = cod\; f $. For each pair $ (f,g) $ of morphisms with $ dom\; g=cod\; f $, a morphism $g\circ f : \dom f \, \rightarrow\,cod\; g$ is the composition $\circ$ and for each object $ a $ there exist a  unique morphism   $  1_A \in {\mathcal C}( A,A ) $ is  called the identity morphism on $a$. Further the composition satisfies $ h \circ ( g \circ f) = (h \circ g) \circ f $ whenever defined and $ f \circ 1_A = f = 1_B \circ  f $ for all $f \in  {\mathcal C} ( A,B) $.

\begin{example}
	\textbf{Set}\,:  objects are sets and morphisms are functions between sets.\\
	\textbf{Grp}\,:\,groups as objects and homomorphisms as morphisms\\
	$\textbf{Vct$ _{\textbf{K}} $} $ :  objects are the vector spaces over a fixed field $ K $ and morphisms are linear maps between them.	
\end{example}

	If a subcollection $ \mathcal {S} $ of objects and morphisms of $ \mathcal {C} $, itselfs consitute a category then $ \mathcal {S} $ is called a subcategory of $ \mathcal {C} $.\\
	
	Let $ \mathcal {C}$ and $ \mathcal {D} $ be two categories.
	A \textit{covariant functor} $F: \mathcal {C}\rightarrow \mathcal {D}$  consists of a vertex map which assigns each $ A\in \nu \mathcal {C}$ to an object $ F(A) \in \nu \mathcal {D}$ and a morphism map which assigns each morphism $f: A\rightarrow B,$ to  a morphism $ F(f):  F(A)\rightarrow  F(B) \in \mathcal D$ such that 	$ F(1_A) = 1_{\nu F(A)}$ for all $A \in \nu \mathcal C$ and $ F(f \circ g) = F(f) \circ F(g) $ for all morphisms $f, g \in \mathcal C$ for which the composition $f \circ g$ exists.

	A functor $F: \mathcal {C}\rightarrow \mathcal {D}$ is said to be \textit{full} if for every pair of objects $ A,B $ in $ \mathcal {C} $ the morphism set $ {\mathcal C} ( A,B)  $ is mapped surjectively by $ F $ onto $ {\mathcal D}(F(A),F(B)) $.
	A subcategory $ \mathcal {S} $ of $ \mathcal {C} $ is said to be \textit{full} if the inclusion functor from $ \mathcal {S} $ to $ \mathcal {C} $ is full.

\begin{definition}
	A morphism $ m : A \rightarrow B$ in a category $\mathcal{C}$ is a \textit{monomorphism} if $f_{1},f_{2} : D \rightarrow A$ in $\mathcal{C} $, the equality  $ m \circ f_{1} = m \circ f_{2} \Rightarrow f_{1}=f_{2}$, that is., $m$ is a monomorphism if it is left cancellable. Dually a morphism $e : A \rightarrow B $ is an \textit{epimorphism} if it is right cancellable. That is if $ g_{1}, g_{2} : b \rightarrow c $ ,\,$g_{1} \circ \,e\,=\,g\,_{2} \circ e  \Rightarrow g_{1}=g_{2}$.\\
\end{definition}

Note that in the category \textbf{Set} monomorphisms are percisely the injections and epimorphisms are precisely the surjections.

\begin{definition}
	An object $ T $ is \textit{terminal} in a category 
	$ {\mathcal C} $ if each object $ A $ in $ {\mathcal C} $ there is exactly one arrow $ A \rightarrow T $. An object $ S $ is \textit{initial}  in $ {\mathcal C} $ if each object $A $ there is exactly one arrow $ S \rightarrow A $. 
\end{definition}
	
	A \textit{zero object} $ 0 $ in a category $ {\mathcal C} $ is an object which is both initial and terminal. For any two objects $ A $ and $ B $ in $ {\mathcal C} $ there is a unique arrow $ 0_{A,B}\,:\, A \rightarrow 0 \rightarrow B  $ called the \textit{zero arrow} from $ A $ to $ B $.
	In the category \textbf{Set}, the empty set is an initial object and any one point set is a terminal object.

\begin{definition}
	Let $ {\mathcal C} $ be a category with zero objects,  \textit{kernel} of a morphism $  f : A \rightarrow B \in  {\mathcal C} $ is a pair $ (K, i) $ of an object $ K $ and a morphism $  i : K \rightarrow A $ such that $ f \circ i = 0 $ satisfying the \textit{universal property}, that is for any other  morphism $  i' : K' \rightarrow A $ with $ f \circ i' = 0 $ there exist a unique arrow $ h : K' \rightarrow K $ such that $ i \circ h = i' $.
\end{definition}

	Dually a \textit{cokernel} of a morphism $  f : A \rightarrow B $ is a pair $ (E, p) $ of an object $ E $ and a morphism $  p : B \rightarrow E $ such that $ p \circ f = 0 $ satisfying the universal property. 

\begin{definition}
	A \textit{product} of two object $ A $ and $ B $ in a category $ {\mathcal C} $ is an object $ A \,\Pi\, B   $ together with morphisms $ p_{1} : A \,\Pi\, B \rightarrow A $ and $ p_{2} : A \,\Pi\, B \rightarrow B $ that satises the
	universal property, viz., for some object $ C $ and any two morphisms $ f_{1} : C \rightarrow A , f_{2} : C \rightarrow B $, there exist a unique morphism $ h : C \rightarrow A \,\Pi\, B $ such that $ p_{i} \circ h = f_{i}  $ for $ i = 1,2 $.
\end{definition}
	
	Dually A \textit{coproduct} of two object $ A $ and $ B $ in a category $ {\mathcal C} $ is an object $ A\, \amalg \,B   $ together with morphisms $ i_{1} : A \rightarrow A\, \amalg \,B $ and $ i_{2} : B \rightarrow A\, \amalg \,B $ that satises the
	universal property: any two morphisms $ f_{1} : A \rightarrow C , f_{2} : B \rightarrow C $ for some object $ C $ there exist a unique morphism $ h : A\, \amalg \,B \rightarrow C $ such that $ h \circ i_{i} = f_{i}  $ for $ i = 1,2 $.\\

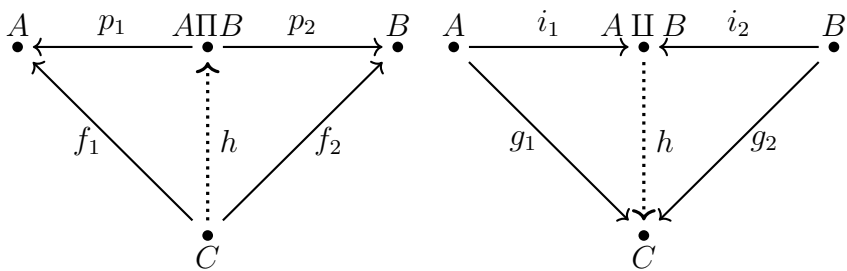
\begin{tikzpicture}
	\fill (2.5,2.5) circle (2pt) node[above] {$B$};
	\fill (0,2.5) circle (2pt) node[above] {$A \Pi B$};
	\fill (-2.5,2.5) circle (2pt) node[above] {$A$};
	\fill (0,0) circle (2pt) node[below] {$C$};
	\draw[->,  thick] (-.2,.2) -- (-2.3,2.3)node[midway, left]{$f_{1}$};
	\draw[->,  thick] (.2,.2) -- (2.3,2.3)node[midway, right]{$f_{2}$};
	\draw[->, dotted, very thick] (0,.2) -- (0,2.3)node[midway, right]{$h$};
	\draw[->,  thick] (-.2,2.5) -- (-2.3,2.5)node[midway, above]{$p_{1}$};
	\draw[->,  thick] (.2,2.5) -- (2.3,2.5)node[midway, above]{$p_{2}$};
\end{tikzpicture}
	\begin{tikzpicture}
	\fill (2.5,2.5) circle (2pt) node[above] {$B$};
	\fill (0,2.5) circle (2pt) node[above] {$A \amalg B$};
	\fill (-2.5,2.5) circle (2pt) node[above] {$A$};
	\fill (0,0) circle (2pt) node[below] {$C$};
	\draw[->,  thick] (-2.3,2.3) -- (-.2,.2)node[midway, left]{$g_{1}$};
	\draw[->,  thick] (2.3,2.3) -- (.2,.2) node[midway, right]{$g_{2}$};
	\draw[->, dotted, very  thick] (0,2.3) -- (0,.2)node[midway, right]{$h $};
	\draw[->,  thick] (-2.3,2.5) -- (-.2,2.5)node[midway, above]{$i_{1}$};
	\draw[->,  thick] (2.3,2.5)--(.2,2.5)  node[midway, above]{$i_{2}$};
\end{tikzpicture}

\begin{definition}
	A category $ {\mathcal C} $ is called \textit{preadditive category} or \textit{$ Ab $-category}  if each hom-set $ {\mathcal C(a,b)} $ is an additive abelian group and composition is bilinear: i.e., 
	$$ (g\,+\,g') \circ (f\,+\,f') = (g \circ f)\,+\, (g \circ f')\,+\, (g' \circ f)\,+\,(g' \circ f')$$
	where $ f,f' : a \rightarrow b \quad\text{and}\quad g,g' : b \rightarrow c$.
\end{definition}
	
		An \textit{additive category} is a preadditive category with a zero object in which every pair of objects admits a product and coproduct and an \textit{abelian category} is a additive category where every morphism admits a kernel and a cokernel, and
		every monomorphism is a kernel and every epimorphism is a cokernel. It is easy to see that the  Category of abelian groups $ \textbf{Ab} $, category of left $ R $ modules $ \textbf{R - Mod} $,  category of right $ R $ modules $ \textbf{Mod - R}$ are abelian categories.	\\
		A morphism $ e : A \rightarrow A $ in the category $ {\mathcal C} $ is called \textit{idempotent} if $ e^{2} = e $. An idempotent  $ e : A \rightarrow A $ is said to be a \textit{split idempotent} if there exist morphisms $ f : B \rightarrow A $ and $ g : A \rightarrow B $ in $ {\mathcal C} $ such that $ g \circ f = 1_{B} $ and $ f \circ g = e $.
\begin{definition}(cf.{\cite{hitchock}})	
	 A category $ {\mathcal C} $ is called \textit{idempotent complete} if all idempotents are spit idempotents.
\end{definition}

	A \textit{preorder} $ \mathcal {P} $ is a category such that for any $ p, p' \in \nu\mathcal {P}, \mathcal{P}(p, p')  $ contains atmost one morphism. In this case there is a quasi order relation $ \subseteq $ on $ \in \nu\mathcal {P} $ such that $ p \subseteq p' \iff \mathcal{P}(p, p') \neq \phi $. $ \mathcal{P} $ is said to be a strict preorder if $ \subseteq $ is a partial order (see .cf.\cite{kss}).

\begin{definition}(cf.\cite{kss}){\label{4}}
	Let $\mathcal{C}$ be a category and $ \mathcal{P} $ be a sub category of $\mathcal{C}$. The pair ($\mathcal{C}$, $\mathcal{P}$) is called \textit{ category with subobjects} if the following conditions hold:
\begin{itemize}
	\item $ \mathcal{P} $ is a strict preorder with $ \nu\mathcal{C} = \nu \mathcal{P} $.
	\item Every $ f \in \mathcal{P} $ is a monomorphism.
	\item If $ f , g \in\mathcal{P} $ and $ f = gh $ for some $ h \in \mathcal{C} $ then $ h \in \mathcal{P} $.
\end{itemize}	
	  Let $ C, D \in \nu\mathcal{C} $, we denote the unique morphism in $ \mathcal{P} $ from $ C \rightarrow D $ by $ j_{(C,D) }$ and is called  \textit{inclusion}. In this case $ C $ is referred to as a \textit{subobject} of $ D $.
\end{definition}
\begin{definition}(cf.\cite{kss})
	Let $ \mathcal{C} $ be a category with subobjects. A \textit{canonical factorization} of a morphism $ f $ in $ \mathcal{C} $ is a factorization of the form $ f = jq $ where $ q $ is an epimorphism and $ j $ is an inclusion.
\end{definition}

\begin{definition}(cf.{\cite{musili}})
	Let $ R $ be a ring. A \textit {left R - module} is an abelian group $ (M,+) $ together with a scalar multiplication $ R \,\times\,M \rightarrow M,\,(r,x) \mapsto rx $ such that:
	\begin{itemize}
		\item $ r(x+y) =rx + ry,\,\, \forall r\,\in R\, \text{and}\, x,y\,\in{M}$
		\item $ (r+r')x = rx +r'x,\,\,\forall r,r'\in R ,x\in{M}$
		\item $ (rr')x = r(r'x) ,\,\,\forall r,r'\in R ,x\in{M} $
	\end{itemize}
\end{definition}
Similary we can define right R module. If R is commutative left R module and right R module are the same.

	\section{Category of left ideals of a Noetherian ring}
Let $ R $ be a Noetherian ring with unity and $ \mathbb{L}_{R} $ be the collection of all left ideals of $ R $. Since  ideals of Noetherian rings are finitely generated, each left ideal in $ \mathbb{L}_{R} $ is of the form $ A = \langle a_{1}, a_{2}, ..., a_{n}\rangle_{l},\quad a_{i} \in R $ for all $i = 1, 2, ..., n$. It is easy to observe that $ \mathbb{L}_{R} $ is a category whose objects left ideals of $ R $ and morphisms are $ R $ -linear transformations.
i.e. for any  $ A, B \in\nu \mathbb{L}_{R}$ and $ f \in \mathbb{L}_{R}(A,B) $, then $ f $  satisfies the conditions
$$	f (x + y) = f (x) + f (y)$$
$$f ( rx) = r f (x)\,\, \forall x , y \in{A} , r\,\in{R}.$$
Since composition of $ R $- linear transformations is again $ R $- linear, the composition of morphisms in the category is the usual set composition of $ R $ linear maps and $ 1_A $
is the identity map on $ A $.

\begin{theorem}
	Let $ R $ be a Noetherian ring with unity. The category $ \mathbb{L}_{R} $, of all left ideals of $ R $ is a preadditive category with zero object.
\end{theorem}

\begin{proof}
	Consider $ A, B \in\nu \mathbb{L}_{R}$ and $ f, g \in \mathbb{L}_{R}(A,B) $. 
	Define $$ (f + g) (x) = f (x) + g (x) \quad\text{for all}\quad x \in {A}$$
	then
\begin{equation*}
	\begin{split}
	(f + g) (x + y) & = f (x + y) + g (x + y ) = f (x) + f (y) + g (x) + g (y)\\
	&=f (x) + g (x) + f (y) + g (y)\\
	& = ( f + g )(x) + ( f + g )(y)
		\end{split}
	\end{equation*}
$$ ( f + g )(rx) = f (rx) + g (rx ) = rf (x) + rg (x ) = r ( f + g )(x)$$	 
that is, $ f + g \in \mathbb{L}_{R}(A,B) $. Since the zero map is $ R $- linear and belongs to	$ \mathbb{L}_{R}(A,B)  $ it is the identity and for each $ f \in \mathbb{L}_{R}(A,B)  $,	let  $ (-f)(x) = - f(x) $ then $ - f \in \mathbb{L}_{R}(A,B) $ and is the invers element. 
	Hence $  \mathbb{L}_{R}(A,B) $ is an Abelian group under above defined addition.	
For any $  f_{1},f_{2} \in \mathbb{L}_{R}(A,B) g_{1},g_{2} \in \mathbb{L}_{R}(B,C) $, $$(g_{1}\,+\,g_{2}) \circ (f_{1}\,+\,f_{2}) = (g_{1} \circ f_{2})\,+\, (g_{1} \circ f_{2})\,+\, (g_{2} \circ f_{1})\,+\,(g_{2} \circ f_{2})$$	
i.e., the composition is bilinear. Hence $ \mathbb{L}_{R} $ is a preadditive category.\\	
 Let $ O $ be the zero ideal. For any $ A  \in \nu \mathbb{L}_{R} $ there is exactly one arrow in $ \mathbb{L}_{R}(A,O)  $ and so $ O $ is the zero object in $ \mathbb{L}_{R} $, that is 
 $ \mathbb{L}_{R} $ is a preadditive category with zero object. 
\end{proof}
This category $ \mathbb{L}_{R} $ is a subcategory of  the category of left  $ R $-modules  $ R-Mod $ and it is easy to see that the inclusion functor $ i : \mathbb{L}_{R} \rightarrow R-Mod $ is full. Similarly it is seen that $ \mathbb{R}_{R} $, the collection of all right ideals of $ R $ is a  preadditive category with zero object and is a full sub category of the category of right  $R$-modules  $Mod - R $.

\begin{theorem}
	Let $ R $ be a Noetherian ring. In the category $ \mathbb{L}_{R} $,  of all left ideals of $ R $ biproduct exist only for ideals with trivial intersection.
\end{theorem}

\begin{proof}
	Let $ A, B \in\nu \mathbb{L}_{R}$ with $ A \cap B = \{0 \} $. Then $ A + B \in\nu \mathbb{L}_{R} $ and since  $ A \cap B = \{0 \} $ every element $ x \in A + B $ can be uniquely expressed as $ x = a + b $ where $ a \in A $ and $ b \in B $. Define $ p_{1}: A + B \rightarrow A $ and $ p_{2}: A + B \rightarrow B $ respectively as $  p_{1} (x) = a $ and $  p_{2} (x) = b$, for all $x = a + b \in A + B$. Clearly $ A + B  $ together with $  p_{1}  $ and $  p_{2}$ constitute the product in left  ideal category $ \mathbb{L}_{R} $. It has the universal property that : for any object $ C \in\nu \mathbb{L}_{R}$ and morphisms  $ f_{1}: C \rightarrow A $ and $ f_{2}: C \rightarrow B $ there exist a unique map $h: C \rightarrow A + B $ as $ h(x) = f_{1}(x) + f_{2}(x)  \forall x \in{C}$ such that the following diagram commutes.
	
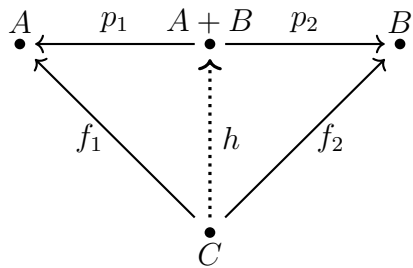
\begin{center}
		
		\begin{tikzpicture}
			\fill (2.5,2.5) circle (2pt) node[above] {$B$};
			\fill (0,2.5) circle (2pt) node[above] {$A + B$};
			\fill (-2.5,2.5) circle (2pt) node[above] {$A$};
			\fill (0,0) circle (2pt) node[below] {$C$};
			\draw[->,  thick] (-.2,.2) -- (-2.3,2.3)node[midway, left]{$f_{1}$};
			\draw[->,  thick] (.2,.2) -- (2.3,2.3)node[midway, right]{$f_{2}$};
			\draw[->, dotted, very thick] (0,.2) -- (0,2.3)node[midway, right]{$h$};
			\draw[->,  thick] (-.2,2.5) -- (-2.3,2.5)node[midway, above]{$p_{1}$};
				\draw[->,  thick] (.2,2.5) -- (2.3,2.5)node[midway, above]{$p_{2}$};
		\end{tikzpicture}
		
	\end{center}
Dually we can define morphism  $ i_{1}: A \rightarrow A + B $	and $ i_{2}: B\rightarrow A + B $ respectively as $  i_{1} (a) = a $ and $  i_{2} (b) = b , \forall a  \in A, \forall b  \in B $ . Then $ A + B  $ together with $  i_{1}  $ and $  i_{2}$ constitute the coproduct in left  ideal category $ \mathbb{L}_{R} $. It has the universal property that : for any object $ D \in\nu \mathbb{L}_{R}$ and morphisms  $ g_{1}: A \rightarrow D $ and $ g_{2}: B \rightarrow D $ there exist a unique map $h': A + B \rightarrow D $ as $ h(x) = g_{1}(a) + g_{2}(b) \,\, \forall x = a + b \in{A + B}$ such that the following diagram commutes.
		
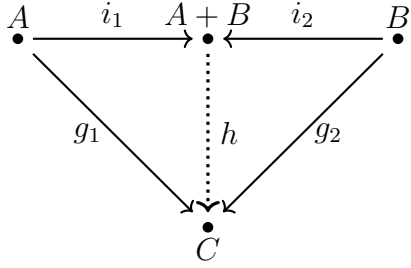
\begin{center}
		
		\begin{tikzpicture}
			\fill (2.5,2.5) circle (2pt) node[above] {$B$};
			\fill (0,2.5) circle (2pt) node[above] {$A + B$};
			\fill (-2.5,2.5) circle (2pt) node[above] {$A$};
			\fill (0,0) circle (2pt) node[below] {$C$};
			\draw[->,  thick] (-2.3,2.3) -- (-.2,.2)node[midway, left]{$g_{1}$};
			\draw[->,  thick] (2.3,2.3) -- (.2,.2) node[midway, right]{$g_{2}$};
			\draw[->, dotted, very  thick] (0,2.3) -- (0,.2)node[midway, right]{$h $};
			\draw[->,  thick] (-2.3,2.5) -- (-.2,2.5)node[midway, above]{$i_{1}$};
			\draw[->,  thick] (2.3,2.5)--(.2,2.5)  node[midway, above]{$i_{2}$};
		\end{tikzpicture}
		
	\end{center}
so in the category $ \mathbb{L}_{R} $  product and coproduct (i.e biproduct) exist only for ideals with trivial intersection.
\end{proof}
The following proposition is recalled as it is of interest in the context of  $ R $-modules categories.

\begin{proposition}(cf.\cite{sam}){\label{1}}
	Let $ \mathcal{C} $ be an additive category and $ \mathcal{D} $ be a full subcategory of $ \mathcal{C} $. If $ \mathcal{D} $ has a zero object and is closed under binary biproduct, then $ \mathcal{D} $ with morphism addition inherited from $ \mathcal{C} $ is an additive category.
\end{proposition}

Since the category $ \mathbb{L}_{R} $ is a full sub category of $ R - Mod $ category and 
$ \mathbb{L}_{R} $ is a preadditive category with zero object, by proposition \ref{1} we can conclude that $ \mathbb{L}_{R} $ is only a preadditive category. 

\begin{theorem}{\label{2}}
	Let $ R $ be a Noetherian ring. Then every  morphism in category $ \mathbb{L}_{R} $   admits a kernel.
\end{theorem}

\begin{proof}
	Let $ f : A \rightarrow B $ be an arrow in $ \mathbb{L}_{R} $. Then $ ker f = \{ x \in{A} : f(x) = 0\} $ is an ideal of $ R $ and $ ker f \in \nu \mathbb{L}_{R} $. Consider the inclusion map $ i : ker f \rightarrow A $. Clearly $ f \circ i = 0 $ and the pair 
	$ (ker f , i) $ is a kernal and admits the universal property,  that for any other pair 
	$ (K , j) $ where $ K $ is an object in $ \mathbb{L}_{R} $ and $ j : K \rightarrow A $ is a morphism with $ f \circ j = 0 $, there exist a unique morphism $ h : K \rightarrow A $ defined by $ h(x) = j (x) \, \forall x \in K $ such that the following diagram commutes.

	
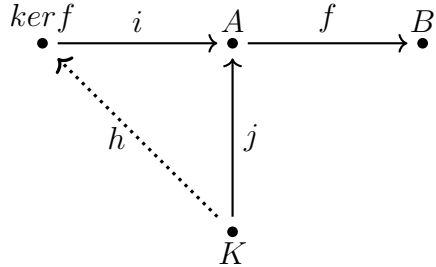
\begin{center}
	
	\begin{tikzpicture}
		\fill (2.5,2.5) circle (2pt) node[above] {$B$};
		\fill (0,2.5) circle (2pt) node[above] {$A$};
		\fill (-2.5,2.5) circle (2pt) node[above] {$ ker f $};
		\fill (0,0) circle (2pt) node[below] {$K$};
		\draw[->, dotted, very  thick] (-.2,.2) -- (-2.3,2.3)node[midway, left]{$ h $};
		\draw[->,  thick] (0,.2) -- (0,2.3)node[midway, right]{$j$};
		\draw[->,  thick] (-2.3,2.5) -- (-.2,2.5)node[midway, above]{$i$};
		\draw[->,  thick] (.2,2.5) --(2.3,2.5)  node[midway, above]{$f$};
	\end{tikzpicture}
	
\end{center}	

\end{proof} 

\begin{theorem}
	Let $ R $ be a Noetherian ring and  $ \mathbb{L}_{R} $ be the category of all left ideals of $ R $. Then only zero map and surjective morphism of $ \mathbb{L}_{R} $  admits a cokernel.
\end{theorem}	
	
\begin{proof}
		Let $ f : A \rightarrow B $ be an arrow in $ \mathbb{L}_{R} $. If $ f = 0 $, then 
		$  B/ f (A)   \cong B $ is an ideal and so $ B/ f (A) \in \nu \mathbb{L}_{R}$. 		  If $ f $ is a surjective map then $  B/ f (A)   \cong $ trivial ideal $ \in \nu \mathbb{L}_{R} $. In these two cases the pair $ ( B/ f (A) , p) $, where $ p : B \rightarrow B/ f (A) $ the usual projection map will give the cokernel. It has the universal property, since for any pair $ (E , q) $ where $ E $ is an object in $ \mathbb{L}_{R} $ and $ q : B \rightarrow E $ is a morphism with $ q \circ f = 0 $,  there exist a unique morphism $ h : B/ f (A) \rightarrow E $ defined by $ h(x) = q (b) ,\, \forall x = b + f (A) \in B/ f (A) $ such that the following diagram commutes.
		
			\begin{center}
			
			\begin{tikzpicture}
				\fill (2.5,2.5) circle (2pt) node[above] {$B / f (A)$};
				\fill (0,2.5) circle (2pt) node[above] {$ B $};
				\fill (-2.5,2.5) circle (2pt) node[above] {$ A $};
				\fill (0,0) circle (2pt) node[below] {$ E $};
				\draw[->, dotted, very  thick] (2.3,2.3) -- (.2,.2)node[midway, left]{$ h $};
				\draw[->,   thick] (0,2.3) -- (0,.2)node[midway, right]{$q$};
				\draw[->,  thick] (-2.3,2.5) -- (-.2,2.5)node[midway, above]{$ f $};
				\draw[->,  thick] (.2,2.5) --(2.3,2.5)  node[midway, above]{$p$};
			\end{tikzpicture}
			
		\end{center}	 
		
\end{proof}	

\begin{lemma}(cf.{\cite{hitchock}}){\label{3}}
	If $ {\mathcal C} $ a preadditive category  then the following are equivalent: 
\begin{enumerate}
	\item $ {\mathcal C} $ idempotent complete.
	\item All idempotents have kernel.
	\item All idempotents have cokerne.
\end{enumerate}
\end{lemma}

\begin{theorem}
		Let $ R $ be a Noetherian ring. In the category $ \mathbb{L}_{R} $,  of all left ideals of $ R $  is idempotent complete
\end{theorem}

\begin{proof}
	We have already proved that in theoerm \ref{2} every morphisms in $ \mathbb{L}_{R} $ have kernel. In particular every idempotent arrows have kernel. Hence by Lemma \ref{3} ,	$ \mathbb{L}_{R} $ is idempotent complete.
\end{proof}

\begin{theorem}
	The category $ \mathbb{L}_{R} $,  of all left ideals of a Noetherian ring $ R $ 
	 is a category with subobjects and every morphisms have canonical factorization.
\end{theorem}
\begin{proof}
	To prove $ \mathbb{L}_{R} $ is a category with subobjects, it will suffices to construct a subcategory $ \mathcal{P} $ of $ \mathbb{L}_{R} $, which satisfies the conditions in Definition \ref{4}. For, dfine a partial order on $ \nu \mathbb{L}_{R}$ as follows:	
	$$ A \subseteq B \iff a_{i} = r_{i1}b_{1} + ... + r_{in}b_{n},\, r_{i1}, ..., r_{in} \in R,\,  i = 1, ..., n $$
	where $ A = < a_{1}, a_{2}, ..., a_{n1} >_{l},\quad B = < b_{1}, b_{2}, ..., b_{n2}>_{l} \in\, \nu \mathbb{L}_{R}$. Then the morphism $ j_{(A, B)}: A \rightarrow B $ defined by $ j_{(A, B)}(x) = x , \,\,\forall x \in{A} $ is a unique monomorphism and 
	the subcategory $ \mathcal{P} $ of $ \mathbb{L}_{R} $ with $ \nu \mathcal{P} = \nu \mathbb{L}_{R} $ and morphisms of $ \mathcal{P} $ are the inclusions$ j_{(A, B)} $  is a strict preorder.\\
	
	Suppose that $ j_{(A, C)} , j_{(B, C)} \in \mathcal{P} $ and $ j_{(A, C)} = j_{(B, C)}h  $  for some $ h \in \mathbb{L}_{R} $.  Then $j_{(A, C)}(x) = j_{(B, C)}h(x)$ for every 
	$ x \in A$ and since both $ j_{(A, C)}$ and $ j_{(B, C)}$ are inclusions we have $ h $ is the inclusion. Hence ($\mathbb{L}_{R}, \mathcal{P}$) is the category with subobject. \\
	Consider any morphism $ f : A \rightarrow B $ in $ \mathbb{L}_{R} $, since $ A $ is left ideal and $ f $ is left $ R $-linear implies $ f(A) $ is a left ideal.
	Let $ q : A \rightarrow f(A) $ be the restriction of $f$ to $im\,(f)$, then it is easy to observe that $ q $ is an epimorphism, $ f(A) \subseteq B $ and $ j_{(f(A), B)} : f(A) \rightarrow B $ is an inclusion. i.e., $ f = j_{(f(A), B)}q $ is a canonical factorization. Thus every morphism in $ \mathbb{L}_{R} $ admits caninical factorization.
\end{proof}

\section{Examples of ideal category of some rings}
In the following we provide some examples of ideal category of some Noetherian rings.
\begin{example} {\textbf{Ideal category of $ \mathbb{Z} $}}
	
		Consider the category $ \mathbb{L_Z} \,[ \mathbb{R_Z}] $ of left [right] ideals of the ring of integers $ \mathbb{Z} $. Then   
		\begin{center}
		$ \nu\mathbb{L_Z} = \{ \langle n\rangle  : n \in \mathbb{Z} \}$ 	\\
		$ \mathbb{L_Z} (\langle n\rangle,\langle m\rangle ) = \{ \rho_{(n,s,m)} : x \mapsto xs\,\vert  ns \in \langle m\rangle\,; \, s \in \mathbb{Z}\,\, \forall x \in \langle n\rangle \}$
		\end{center}
Now $ \rho_{(n,s,m)} :\langle n\rangle \,\rightarrow \,\langle m\rangle  $ and $ \rho_{(m,t,p)} : \langle m\rangle\, \rightarrow \, <p>  $	and thier composition is  $ \rho_{(m,t,p)}\circ \rho_{(n,s,m)} = \rho_{(n,st,p)} :  \langle n\rangle \,\rightarrow \,<p>  $	. For 
$ \rho_{(n,s,m)} , \rho_{(n,t,m)} \in  \mathbb{L_Z} (\langle n\rangle \,,\,\langle m\rangle)$, 
let $ \rho_{(n,s,m)} + \rho_{(n,t,m)} = \rho_{(n,s + t,m)}$, with respect to this addition $ \mathbb{L_Z} (\langle n\rangle \,, \,\langle m\rangle) $ is an abelian group and $ \langle 0\rangle $ is the zero element, hence $ \mathbb{L_Z} $ is a preadditive category with zero object.\\
For any two non zero ideals $ \langle n\rangle \,, \,\langle m\rangle $ of $ \mathbb{Z} $, the element $ mn \ $ always belongs to $ \langle n\rangle \cap \,\langle m\rangle $, and so  
$ \langle n\rangle \cap \langle m\rangle = \{0\} $  only when  $ n = 0 $ or $ m = 0 $.
Hence in  $ \mathbb{L_Z} $ biproduct exist only for those pair of objects in which one of them is the zero ideal. A morphism $ \rho_{(n,s,m)} :\langle n\rangle \,\rightarrow \,\langle m\rangle  $ is a monomorphism for $ s \neq 0 $ and zero object together with zero arrow will give the kernel of this morphism.
\end{example}

\begin{example}{\textbf{Ideal category of $ \mathbb{Z}_{6} $}}
	
Let $  R = \mathbb{Z}_{6} $ and $ \mathbb{L}_{R} $ be  the category whose objects are left ideals of the ring  and morphisms are $ R $ - linear transformations, i.e.,

\begin{center}  
$ \nu\mathbb{L_Z}_{6} = \{ \langle 0\rangle ,\langle 1\rangle ,\langle 2\rangle,\langle 3\rangle\}$ 	\\
$ \mathbb{L_Z}_{6} (\langle n\rangle \,, \,\langle m\rangle) = \{ \rho_{(n,s,m)} : x \mapsto xs\,\vert \, ns \in \langle m\rangle\,; \, s \in \mathbb{Z}_{6}\,\, \forall x \in \langle n\rangle \}$\\	
\end{center}

The composition and addtion is defined as in the case of $ \mathbb{Z} $. Hence 
$ \mathbb{L_Z}_{6} $ is a preadditive category with zero object $ \langle 0\rangle $. 
In $ \mathbb{L_Z}_{6} $ biproduct exisit for the pairs $ (\langle 2\rangle,\langle 3\rangle)$ and to $ (\langle 0\rangle , \langle n\rangle) $ where $ n = 1,2,3 $.
\end{example}

	\end{document}